\documentclass[11pt]{article}

\usepackage[T1]{fontenc}
\usepackage[utf8]{inputenc}
\usepackage{lmodern}
\usepackage{microtype}
\usepackage[margin=1in]{geometry}
\usepackage{amsmath,amssymb,amsthm,mathtools}
\usepackage{enumitem}
\usepackage[dvipsnames]{xcolor}
\usepackage{hyperref}
\usepackage{mathrsfs}
\usepackage[nameinlink,noabbrev]{cleveref}

\hypersetup{
  colorlinks=true,
  linkcolor=blue!50!black,
  citecolor=blue!50!black,
  urlcolor=blue!50!black,
  pdftitle={On the Codimension-1 PGL_4 Orbit Closures in Gr(2,10)},
  pdfauthor={Ari Krishna}
}

\newtheorem{theorem}{Theorem}[section]
\newtheorem{proposition}[theorem]{Proposition}
\newtheorem{lemma}[theorem]{Lemma}
\newtheorem{corollary}[theorem]{Corollary}
\newtheorem{definition}[theorem]{Definition}
\newtheorem{remark}[theorem]{Remark}

\crefname{theorem}{theorem}{theorems}
\Crefname{theorem}{Theorem}{Theorems}
\crefname{lemma}{lemma}{lemmas}
\Crefname{lemma}{Lemma}{Lemmas}
\crefname{proposition}{proposition}{propositions}
\Crefname{proposition}{Proposition}{Propositions}
\crefname{corollary}{corollary}{corollaries}
\Crefname{corollary}{Corollary}{Corollaries}
\crefname{definition}{definition}{definitions}
\Crefname{definition}{Definition}{Definitions}
\crefname{remark}{remark}{remarks}
\Crefname{remark}{Remark}{Remarks}

\newcommand{\G}{\mathrm{Gr}}
\newcommand{\Pj}{\mathbb P}

\newcommand{\Sym}{\mathrm{Sym}}

\newcommand{\Sing}{\mathrm{Sing}}
\newcommand{\PGL}{\mathrm{PGL}}
\newcommand{\GL}{\mathrm{GL}}

\title{On the Codimension-1 \texorpdfstring{$\PGL_4$}{PGL4} Orbit Closures in \texorpdfstring{$\G(2,10)$}{Gr(2,10)}}
\author{Ari Krishna}
\date{}

\begin{document}
\maketitle

\begin{abstract}
We study the natural action of $\PGL(V)$ on the Grassmannian
\[
G:=\G(2,\Sym^2 V^\vee), \qquad \dim V=4,
\]
whose points are pencils of quadrics in $\Pj(V)\cong \Pj^3$. Here $\dim G=16$ while $\dim \PGL(V)=15$, so the generic orbit has codimension one and one expects a one-parameter family of generic orbits. We construct this family via the $j$-invariant of the discriminant binary quartic of a pencil. We then determine the codimension-one orbit closures and compute their Chow classes.
The smooth codimension-one orbit closures are the reduced fibers of the $j$-map on the smooth
locus, while the unique boundary divisor is the closure of the orbit of a nodal quartic complete
intersection of arithmetic genus $1$ and geometric genus $0$. Every divisorial fiber of the rational $j$-map has class
$12\sigma_1$ in $A^1(G)$. For the reduced codimension-one orbit closures one has
$[\overline{O_a}]=12\sigma_1$ for $a\neq 0,1728,\infty$,
$[\overline{O_{1728}}]=6\sigma_1$, $[\overline{O_0}]=4\sigma_1$, and $[T]=12\sigma_1$.
\end{abstract}

\tableofcontents

\section{Introduction}

Let $V$ be a $4$-dimensional vector space over an algebraically closed field $k$ of characteristic $0$, and let
\[
G:=\G(2,\Sym^2 V^\vee).
\]
A point of $G$ is a $2$-dimensional linear system of quadrics in $\Pj(V)\cong \Pj^3$, i.e. a pencil of quadrics. The group $\PGL(V)$ acts naturally on $G$, and
\[
\dim G = 2(10-2)=16, \qquad \dim \PGL(V)=4^2-1=15.
\]
Consequently, the generic orbit has codimension one in $G$, so we expect a one-parameter family of generic orbits.

The purpose of this paper is to actualize that expectation, and to compute the Chow classes of the resulting codimension-one orbit closures. The following construction is central to our approach: a pencil of quadrics
\[
W=\langle Q_0,Q_1\rangle \in G
\]
has an associated discriminant binary quartic
\[
\Delta_W(s,t):=\det(sQ_0+tQ_1),
\]
whose roots record the singular members of the pencil. On the open locus where $\Delta_W$ has four distinct roots, this root divisor determines the $\PGL(V)$-orbit; equivalently, the generic orbit is classified by the corresponding point of the (coarse) moduli space of unordered quadruples on $\Pj^1$, i.e. by the $j$-invariant of the discriminant binary quartic.

Our computations are parallel to Choudhary's computations of the codimension-one orbit closures for the $\PGL_3$-action on $\G(3,V_2)$ in \cite{Choudhary}, since the same codimension-one-orbit mechanism underlies each. However, one geometric difference is that our special fiber has a unique divisorial component, whereas Choudhary's special fiber splits into two codimension-one orbit closures with multiplicities.

Our main result is the following.

\begin{theorem}\label{thm:main}
Let $G=\G(2,\Sym^2 V^\vee)$ with $\dim V=4$, and let $\PGL(V)$ act naturally on $G$. Then:
\begin{enumerate}[label=\textup{(\roman*)}]
    \item There is a rational $\PGL(V)$-invariant map
    \[
    p\colon G \dashrightarrow \Pj^1
    \]
    whose restriction to the open locus of smooth pencils is given by the $j$-invariant of the discriminant binary quartic.
    \item For each finite value $a\in \Pj^1\setminus\{\infty\}$ lying in the image of the smooth-locus map $p\colon U_{\mathrm{sm}}\to \Pj^1$, the reduced fiber
    \[
    O_a := \bigl(p^{-1}(a)\cap U_{\mathrm{sm}}\bigr)_{\mathrm{red}}
    \]
    is a single codimension-one $\PGL(V)$-orbit, and its closure $\overline{O_a}\subset G$
    is a codimension-one orbit closure.
    \item The special fiber $F_{\infty}$ has a unique divisorial component, namely the closure of the orbit of the pencil
    \[
    W_{\mathrm{node}}:=\big\langle x_0^2+x_1^2+x_2^2,\ x_0x_3+x_1^2+2x_2^2\big\rangle.
    \]
    The corresponding complete intersection in $\Pj^3$ is a quartic complete intersection of arithmetic genus $1$ with one ordinary node; in particular, its normalization has geometric genus $0$.
    \item Every divisorial fiber of the rational map $p$ has Chow class $12\sigma_1$ in
    $A^1(G)$. For the reduced codimension-one orbit closures one has
    \[
    [\overline{O_a}] = 12\sigma_1 \quad (a\notin\{0,1728,\infty\}),
    \qquad
    [\overline{O_{1728}}]=6\sigma_1,
    \qquad
    [\overline{O_0}]=4\sigma_1,
    \qquad
    [T]=12\sigma_1.
    \]
    Moreover,
    \[
    F_{1728}=2\,\overline{O_{1728}},
    \qquad
    F_0=3\,\overline{O_0}
    \]
    scheme-theoretically.
\end{enumerate}
\end{theorem}

The proof proceeds in three steps.
\begin{enumerate}[label=\arabic*.]
    \item We classify the generic smooth orbits by simultaneous diagonalization and the discriminant divisor.
    \item We compute the class of a general fiber by intersecting with a complementary Schubert cycle and reducing to a degree computation for the $j$-map attached to lines through a point in a plane quartic.
    \item We analyze the special fiber $j=\infty$ by identifying it with the repeated-root locus for discriminant quartics and proving that its unique divisorial component is the generic tangency locus for the determinant quartic hypersurface.
\end{enumerate}

\subsection*{Acknowledgements}
This paper was inspired by Choudhary's paper \cite{Choudhary}, which continued a program initiated by Goel in \cite{Goel}. I am thankful to the two of them for their time and advice. I am deeply grateful to Prof. Joe Harris for his continued mentorship and generosity; I owe essentially everything I know about algebraic geometry to him. 

\section{The discriminant quartic and the \texorpdfstring{$j$}{j}-map}
Our first task is to replace the pencil by its inscription on $\mathbb{P}^1$.

Let
\[
\Pj^9 := \Pj(\Sym^2 V^\vee)
\]
be the projective space of quadrics in $\Pj(V)\cong \Pj^3$, and let
\[
D:=\{[Q]\in \Pj^9 : \det(Q)=0\}
\]
be the determinant quartic hypersurface of singular quadrics.

A point $W\in G$ determines a line
\[
\ell_W:=\Pj(W)\subseteq \Pj^9.
\]
If $W=\langle Q_0,Q_1\rangle$, then the scheme-theoretic intersection $\ell_W\cap D$ is cut out by the binary quartic
\[
\Delta_W(s,t):=\det(sQ_0+tQ_1).
\]
By construction, its zero-locus divisor on $\Pj^1$ records the singular members of the pencil.

%return to fix wording below! 

\begin{definition}
Let $U_{\mathrm{sm}}\subseteq G$ be the open locus of pencils for which $\Delta_W$ has four distinct roots. Equivalently, $U_{\mathrm{sm}}$ is the locus of pencils whose complete intersection in $\Pj^3$ is smooth.
\end{definition}

The equivalence in the definition is classical. If $X_W\subseteq \Pj^3$ is singular, then some point of $X_W$ is a common zero of $Q_0,Q_1$ and of all their first derivatives, so some singular quadric in the pencil has its vertex on the base locus; equivalently, the discriminant acquires a repeated root. 

Conversely, if $\Delta_W$ has a repeated root, then the pencil contains a rank-$3$ quadric $Q$ for which the tangent direction in the pencil annihilates the vertex of $Q$; equivalently, the vertex lies on the other generator of the pencil and hence on the base locus. The Jacobian criterion then shows that $X_W$ is singular at that vertex. Equivalently, the smooth complete intersections of two quadrics are exactly the pencils with Segre symbol $[1,1,1,1]$; see the discussion of pencils of quadrics and Segre symbols in \cite{Dolgachev}.

We classify the $\PGL(V)$-orbits on $U_{\mathrm{sm}}$, which amounts to some linear algebra. 

\begin{proposition}[Simultaneous diagonalization]\label{prop:diag}
Let $W\in U_{\mathrm{sm}}$. Then $W$ is $\PGL(V)$-equivalent to a pencil of the form
\[
W_\lambda=
\left \langle
x_0^2+x_1^2+x_2^2+x_3^2,\
\lambda_0x_0^2+\lambda_1x_1^2+\lambda_2x_2^2+\lambda_3x_3^2
\right\rangle
\]
with $\lambda_0,\lambda_1,\lambda_2,\lambda_3$ pairwise distinct.
\end{proposition}

\begin{proof}
Choose a smooth quadric $Q_0\in W$ (since we are in $U_{\mathrm{sm}}$, such a member exists). Let $B_0,B_1$ be the symmetric matrices of $Q_0,Q_1$ in some chosen basis of $V$. Because $Q_0$ is smooth, $B_0$ is invertible. Define the linear map
\[
T:=B_0^{-1}B_1\in \operatorname{End}(V).
\]
Then, $T$ is $B_0$-self-adjoint, because
\[
B_0T=B_1=B_1^\top=T^\top B_0.
\]
Moreover,
\[
\Delta_W(s,t)=\det(sB_0+tB_1)=\det(B_0)\det(sI+tT).
\]
Thus, the roots of $\Delta_W$ are the negatives of the eigenvalues of $T$. Since $\Delta_W$ has four distinct roots, $T$ has four distinct eigenvalues, so it is diagonalizable. Because $T$ is self-adjoint with respect to the nondegenerate symmetric form $B_0$, its eigenspaces for distinct eigenvalues are pairwise $B_0$-orthogonal. Each eigenline is nonisotropic, which is to say that if $B_0(v_i,v_i)=0$ for some eigenvector $v_i$, then $v_i$ would be orthogonal to every eigenspace and would hence lie in the radical of $B_0$, which is impossible by nondegeneracy. Now, choosing a $B_0$-orthogonal eigenbasis $\{v_0,\dots,v_3\}$ and rescaling each vector so that $B_0(v_i,v_i)=1$, we obtain coordinates in which
\[
Q_0=x_0^2+x_1^2+x_2^2+x_3^2,
\qquad
Q_1=\lambda_0x_0^2+\lambda_1x_1^2+\lambda_2x_2^2+\lambda_3x_3^2.
\]
This is the desired form.
\end{proof}

\begin{proposition}[Orbit classification on the smooth locus]\label{prop:orbit-classification}
Two pencils in $U_{\mathrm{sm}}$ lie in the same $\PGL(V)$-orbit if and only if their discriminant divisors on $\Pj^1$ are projectively equivalent. Equivalently, the generic orbit is determined by the $j$-invariant of the discriminant binary quartic.
\end{proposition}

\begin{proof}
If $g\in \PGL(V)$ carries $W$ to $W'$, then after choosing representatives in $\GL(V)$, we has
\[
\Delta_{W'}(s,t)=c\,\Delta_W(s,t)
\]
for some nonzero scalar $c$, so the two discriminant divisors are projectively equivalent.

Conversely, Proposition~\ref{prop:diag} permits us to assume that both pencils are diagonal. The projective classification of smooth pencils of quadrics in $\Pj^3$ asserts that a pencil with Segre symbol $[1,1,1,1]$ is determined up to $\PGL(V)$ by the projective equivalence class of its discriminant divisor on the parameter line (see \cite[\S8.2]{Dolgachev} and Reid's discussion of complete intersections of two quadrics \cite{Reid}). Equivalently, two smooth diagonal pencils are $\PGL(V)$-equivalent if and only if their unordered four-point discriminant divisors are $\PGL_2$-equivalent. This proves the claim.
\end{proof}

\begin{definition}
Let
\[
p\colon U_{\mathrm{sm}}\to \Pj^1
\]
be the regular map sending $W$ to the $j$-invariant of the discriminant binary quartic $\Delta_W$. We continue to denote by
\[
p\colon G\dashrightarrow \Pj^1
\]
the induced rational map.
\end{definition}

This map $p$ is dominant. To this end, for $\lambda\in k\setminus\{0,1\}$, consider the squarefree binary quartic
\[
f_\lambda(s,t)=st(s-t)(s-\lambda t).
\]
The double cover $y^2=f_\lambda(s,t)$ is the elliptic curve $y^2=x(x-1)(x-\lambda)$ after dehomogenizing by setting $t=1$; thus, the binary-quartic $j$-invariant agrees with the elliptic-curve $j$-invariant, and is therefore given by
\[
j(\lambda)=2^8 \,\frac{(1-\lambda+\lambda^2)^3}{\lambda^2(1-\lambda)^2};
\]
see \cite[Chapter~III]{Silverman}. This is a nonconstant rational function on $\Pj^1$, hence surjective because $k$ is algebraically closed. 

The identification with the elliptic-curve $j$-invariant comes from the double cover of $\Pj^1$ branched at the four roots of the binary quartic. Moreover, every squarefree binary quartic occurs as the discriminant of a smooth diagonal pencil: indeed, if
\[
f(s,t)=\prod_{i=0}^3 (s-\lambda_i t)
\]
has four distinct roots, then the diagonal pencil
\[
\left\langle \sum_{i=0}^3 x_i^2,\ \sum_{i=0}^3 \lambda_i x_i^2\right\rangle
\]
has discriminant proportional to $f$. It follows that $p$ is surjective on the smooth locus and is therefore dominant as a rational map on $G$.

\begin{corollary}\label{cor:fibers-are-orbits}
For every $a\in \Pj^1\setminus\{\infty\}$ corresponding to a smooth discriminant quartic,
the set-theoretic fiber $p^{-1}(a)\cap U_{\mathrm{sm}}$ is a single $\PGL(V)$-orbit $O_a$.
Its closure $\overline{O_a}\subset G$ is an irreducible codimension-one orbit closure.
Moreover, no extra divisorial component enters when taking the closure in $G$.
The scheme-theoretic multiplicity of $F_a$ along $\overline{O_a}$ is determined in Section~6.
\end{corollary}

\begin{proof}
By Proposition~\ref{prop:orbit-classification}, the set-theoretic fiber of $p$ on $U_{\mathrm{sm}}$ is a single
$\PGL(V)$-orbit, hence it is irreducible. By Proposition~\ref{prop:finite-stabilizer-smooth}, every smooth orbit has dimension $15$, so its closure
has codimension one in $G$.
As we will prove in Theorem~\ref{thm:special-fiber-structure}, the only divisorial boundary in
$G\setminus U_{\mathrm{sm}}$ lies over $\infty$, so no extra divisorial component enters for
finite $a$.
\end{proof}

\section{Intersection theory and the general fiber class}

In this section, we compute the class of a smooth fiber $F_a$ in $A^1(G)$. We begin by recalling the relevant intersection-theoretic machinery for Grassmannians.

\subsection{Schubert calculus on the Grassmannian}\label{subsec:schubert}

We briefly recall the intersection theory of Grassmannians, following \cite[Chapters~4--6]{EH}; the reader is referred there for proofs and further details.

Let $\G(k,n)$ denote the Grassmannian of $k$-dimensional linear subspaces of an $n$-dimensional vector space, or equivalently the space of $(k-1)$-dimensional projective linear subspaces of $\Pj^{n-1}$. In our setting, $G=\G(2,\Sym^2 V^\vee)\cong \G(2,10)$, so points of $G$ are $2$-planes in a $10$-dimensional vector space; equivalently, they are lines in $\Pj^9$.

The Chow ring $A^*(G)$ is generated as an abelian group by the \emph{Schubert classes} $\sigma_a$ and more generally $\sigma_{a_1,\ldots,a_s}$, which are defined as follows. Fix a complete flag
\[
0 = V_0 \subset V_1 \subset V_2 \subset \cdots \subset V_n = \Sym^2 V^\vee.
\]
Given a partition $\lambda=(a_1\geq a_2\geq \cdots \geq a_k\geq 0)$ with $a_1\leq n-k$, its \emph{Schubert cycle} is
\[
\Sigma_\lambda := \{ W \in \G(k,n) : \dim(W \cap V_{n-k+i-a_i}) \geq i, \quad 1\leq i\leq k \},
\]
and the corresponding \emph{Schubert class} $\sigma_\lambda=[\Sigma_\lambda]\in A^{|\lambda|}(G)$ depends only on the partition $\lambda$, not on the choice of flag (\cite[Theorem~4.2]{EH}). Here, $|\lambda|=\sum a_i$ is the codimension of $\Sigma_\lambda$ in $G$.

For the Grassmannian $\G(k,n)$, the Chow ring is
\[
A^*(\G(k,n)) \cong \mathbb{Z}[\sigma_1,\ldots,\sigma_{n-k}] / I,
\]
where $I$ is the ideal generated by relations coming from the Whitney sum formula for Chern classes of the tautological sequence (\cite[Chapter~5]{EH}). The group $A^1(\G(k,n))$ is freely generated by the \emph{special Schubert class} $\sigma_1$, which is the class of the hyperplane section under the Pl\"ucker embedding. Equivalently, $\operatorname{Pic}(G)\cong \mathbb{Z} \langle \sigma_1 \rangle$ (\cite[Proposition~4.3]{EH}).

In our case, $G=\G(2,10)$ has dimension $2(10-2)=16$, and the relevant intersection pairing for divisor-class computations is
\[
A^1(G) \times A^{15}(G) \to A^{16}(G) \cong \mathbb{Z}.
\]
The top-dimensional Schubert class in $A^{15}(G)$ is $\sigma_{8,7}$, which is the class of a point-pencil of lines in a fixed plane of $\Pj^9$. Then, the crucial intersection number is
\[
\sigma_1 \cdot \sigma_{8,7} = 1,
\]
which follows from the Pieri formula (\cite[Theorem~4.7]{EH}): $\sigma_1\cdot \sigma_{a,b}=\sum \sigma_{a',b'}$, where the sum is over partitions obtained by adding one box to $(a,b)$ while remaining a valid partition for $\G(2,10)$. In the extreme case $\sigma_{8,7}$, the only valid partition with $|(a',b')|=16=\dim G$ is $(8,8)$, and $\sigma_{8,8}$ is the class of a point; hence $\sigma_1\cdot\sigma_{8,7}=\sigma_{8,8}$, which has degree $1$. As a consequence, if $[D]=m\sigma_1$ is the class of a divisor, then $m$ can be recovered by intersecting with the complementary cycle $\sigma_{8,7}$, i.e.
\[
m = [D] \cdot \sigma_{8,7},
\]
which we may try to compute directly via enumerative techniques. This is the strategy we will use to compute the class of a fiber of the $j$-map.

\subsection{The complementary Schubert cycle}

We use the identification
\[
\G(2,\Sym^2V^\vee)\cong \G(1,\Pj(\Sym^2V^\vee))=\G(1,\Pj^9),
\]
which sends a $2$-plane in $\Sym^2V^\vee$ to its associated line in $\Pj^9$.

Fix a general plane
\[
\Pi\cong \Pj^2\subseteq \Pj^9
\]
and a general point $q\in \Pi$. Let $B\cong \Pj^1$ be the family of lines in $\Pi$ passing through $q$. Viewed inside $G=\G(1,\Pj^9)$, this is the Schubert cycle $\sigma_{8,7}$ consisting of lines contained in the fixed plane $\Pi$ and passing through the fixed point $q$. As discussed in Section~\ref{subsec:schubert},
\[
\sigma_1\cdot \sigma_{8,7}=1.
\]
Let
\[
C:=\Pi\cap D,
\]
which is a quartic plane curve. For general $\Pi$, we have that $C$ is smooth, and we choose $q\in \Pi\setminus C$. Every line $\ell\in B$ meets $C$ in a degree-$4$ divisor, which is precisely the discriminant divisor of the corresponding pencil of quadrics.

\subsection{The induced \texorpdfstring{$j$}{j}-map on the Schubert slice}

To compute the divisor class, we pass to the microcosm of a single Schubert slice. There, the determinant hypersurface becomes a plane quartic, and the orbit problem reduces to the geometry of tangent lines through a point. 

First, projection from $q$ furnishes a finite morphism
\[
\pi_q\colon C\to B\cong \Pj^1
\]
of degree $4$. Note that the branch points of $\pi_q$ correspond exactly to the lines through $q$ tangent to $C$, which we enumerate below. 

\begin{lemma}\label{lem:12-tangents}
There are exactly $12$ simple tangent lines to $C$ through $q$.
\end{lemma}

\begin{proof}
Since $C$ is a smooth plane quartic, it has genus $$g(C) = \dfrac{(4-1)(4-2)}{2} = 3. $$ Applying the Riemann--Hurwitz formula to the degree-$4$ map $\pi_q$ gives
\[
2g(C)-2 = 4\bigl(2g(\Pj^1)-2\bigr)+\deg R,
\]
so $4 = -8 + \deg R$, i.e. $\deg R = 12$.  
For general $q$, all ramification is simple, so there are exactly $12$ simple tangent lines through $q$.
\end{proof}

The $j$-invariant of binary quartics is a rational function whose polar divisor is the discriminant hypersurface. Restricting it to the Schubert slice $B$, we acquire a rational function whose poles occur at tangent lines to $C$. Since $C$ is a smooth plane quartic, no line in $B$ is contained in $C$, so this restriction is indeed a well-defined rational map $B\dashrightarrow \Pj^1$.

\begin{proposition}\label{prop:degree-on-slice}
The restriction of the $j$-map to the Schubert slice $B$ has degree $12$. Consequently,
\[
[F_a]\cdot \sigma_{8,7}=12.
\]
\end{proposition}

\begin{proof}
By construction, a point $\ell\in B$ gives a binary quartic with repeated root if and only if $\ell$ is tangent to $C$. We need to understand the local behavior of the $j$-map near such a tangent line.

We invoke invariant theory for binary quartics (see \cite[Chapter~III, \S1]{Silverman} and \cite[\S6.3]{Dolgachev}).
A binary quartic $f(s,t)=\sum \binom{4}{i}a_is^{4-i}t^i$ has two fundamental $\mathrm{SL}_2$-invariants:
\[
I(f) = a_0a_4 - 4a_1a_3 + 3a_2^2, \qquad
J(f) = a_0a_2a_4 + 2a_1a_2a_3 - a_2^3 - a_0a_3^2 - a_1^2a_4,
\]
of degrees $2$ and $3$ respectively, and the discriminant
\[
\Delta(f) = I^3 - 27 J^2.
\]
(Our sign convention follows \cite[\S6.3.4]{Dolgachev}.) The $j$-invariant of the binary quartic (equivalently, the $j$-invariant of the elliptic curve given by the double cover of $\Pj^1$ branched at the four roots of $f$) is the ratio
\begin{equation}\label{eq:j-binary-quartic}
j(f) = 1728\,\frac{I^3}{\Delta} = 1728\,\frac{I^3}{I^3 - 27J^2}.
\end{equation}
This is a well-defined rational function on the space of binary quartics with nonvanishing determinant, i.e. with four distinct roots. The identification with the elliptic-curve $j$-invariant is classical; see \cite[Proposition~III.1.7]{Silverman}.

Now, consider a simple tangent line $\ell_0\in B$. Near $\ell_0$, the family of binary quartics obtained by restricting the determinantal quartic to lines through $q$ varies algebraically with the parameter on $B$. Let $u$ be a local coordinate on $B$ centered at $\ell_0$. Since $\ell_0$ is a simple tangent to $C$, exactly one pair of the four intersection points of $\ell$ with $C$ coalesces transversely as $u\to 0$: the local intersection multiplicity of $\ell_0$ with $C$ at the tangency point is $2$, while the remaining two intersection points stay distinct and away from the tangency. In terms of the discriminant, this means that $\Delta$ has a simple zero at $u=0$: the discriminant of a one-parameter family of degree-$4$ polynomials vanishes to order $1$ at a generic simple tangency.

We readily see that the condition $I(f)=0$ is a closed condition of codimension $1$ in the space of binary quartics. Restricted to the $1$-parameter family parametrized by $B$, it is satisfied at only finitely many points. Moreover, for general $\Pi$ and $q$, we may invoke Kleiman transversality \cite{Kleiman} to ensure that none of the $12$ simple tangent lines belongs to this finite set. Consequently, at each simple tangency $\ell_0$, we have $I\neq 0$, and since $\Delta$ vanishes to order $1$ while $I^3$ does not vanish, the expression \eqref{eq:j-binary-quartic} shows that $j$ has a simple pole at $u=0$.

By Lemma~\ref{lem:12-tangents}, there are exactly $12$ such simple tangent lines, and for general $\Pi$ and $q$, these are the only poles of $j|_B$. Since a rational function on $\Pj^1$ has an equal number of poles and zeros when counted with multiplicities, $j|_B$ has degree $12$. Because $B$ represents the complementary Schubert cycle $\sigma_{8,7}$ (see Section~\ref{subsec:schubert}), the intersection number $[F_a]\cdot \sigma_{8,7}$ counts the number of points of $B$ lying in the fiber $F_a$, which is precisely the number of preimages of a general value $a$ under $j|_B$. This is the degree of the restricted map, so it equals $12$.
\end{proof}

\begin{remark}
We may also recover this count from the Pl\"ucker formula as follows: from a general point of the plane, there are $d(d-1)-2\delta-3\kappa=12$ tangent lines to a smooth plane quartic ($d=4$, $\delta=\kappa=0$). 
\end{remark}

Now, we are equipped to compute the class of a general fiber. 

\begin{theorem}\label{thm:general-fiber-class}
For every smooth value $a\in \Pj^1\setminus\{\infty\}$, the
scheme-theoretic fiber satisfies
\[
[F_a]=12\sigma_1
\]
in $A^1(G)$.
\end{theorem}

\begin{proof}
Write $[F_a]=m\sigma_1$. Intersecting with the complementary Schubert cycle $\sigma_{8,7}$ and using Proposition~\ref{prop:degree-on-slice} gives
\[
m = [F_a]\cdot \sigma_{8,7}=12.
\]
Thus, $[F_a]=12\sigma_1$.
\end{proof}

\section{The special fiber and the tangency divisor}

We now turn our attention to the special fiber $p=\infty$. Set-theoretically, this is the repeated-root locus for discriminant quartics.

\begin{definition}
Let $R\subseteq G$ be the locus of pencils whose discriminant quartic has a repeated root. Let
\[
F_\infty:=\overline{R}\subseteq G.
\]
Define the tangency locus
\[
T:=\overline{\{W\in G : \ell_W \text{ is tangent to } D \text{ at a smooth point, and otherwise transverse to } D\}}.
\]
\end{definition}

The core of this section is to prove that $T$ is the unique divisorial component of $F_\infty$.

\subsection{A normal form for the generic tangency pencil}

First, we put a general point of $T$ into a convenient normal form.

\begin{lemma}\label{lem:normal-form}
Let $W\in T$ be a general point. Then, $W$ is $\PGL(V)$-equivalent to a pencil of the form
\[
W_{a,b}:=
\left \langle
x_0^2+x_1^2+x_2^2,\
 x_0x_3 + a x_1^2 + b x_2^2
\right \rangle
\]
with $a,b\in k^\times$ and $a\neq b$. Its discriminant quartic has root type $2+1+1$, namely,
\[
\Delta_{W_{a,b}}(1,t) = -\frac14 t^2(1+at)(1+bt).
\]
\end{lemma}

\begin{proof}
Let $W=\langle Q_0,Q_1\rangle$ with $\ell_W$ tangent to $D$ at a smooth point $[Q_0]\in D$. Since $[Q_0]\in D_{\mathrm{sm}}$, the quadric $Q_0$ has rank $3$, and after a projective change of coordinates if necessary, we may assume that
\[
Q_0=x_0^2+x_1^2+x_2^2.
\]
Thus, $\ker(Q_0)=\langle e_3\rangle$.

Write
\[
Q_1=L(x_0,x_1,x_2)x_3 + q(x_0,x_1,x_2) + c x_3^2.
\]
The condition of tangency of the line $\ell_W$ to $D$ at $[Q_0]$ means that the linear term of
$\det(Q_0+tQ_1)$
in $t$ vanishes. Writing $A=\operatorname{diag}(1,1,1,0)$ for the matrix of $Q_0$, we have
\[
\frac{d}{dt}\det(A+tB)\big|_{t=0}=\operatorname{tr}(\operatorname{adj}(A)B),
\]
and here $\operatorname{adj}(A)$ is the matrix with a single $1$ in the $(4,4)$-entry. Therefore, the tangency condition translates exactly to the vanishing of the $x_3^2$-coefficient of $Q_1$, i.e. $c=0$. Moreover, $L\neq 0$, as otherwise every quadric in the pencil would be singular at $[0:0:0:1]$ and the line would lie in $D$.

By using the orthogonal group of the quadratic form $x_0^2+x_1^2+x_2^2$, we may arrange that $L=x_0$. Replacing $x_3$ by $x_3+\ell(x_0,x_1,x_2)$ leaves $Q_0$ unchanged, and it transforms
\[
Q_1=x_0x_3+q
\]
into
\[
x_0x_3+(x_0\ell+q).
\]
Choosing $\ell=-\alpha x_0-\beta x_1-\gamma x_2$ cancels the $x_0^2$, $x_0x_1$, and $x_0x_2$ terms of $q$. Finally, diagonalizing the remaining quadratic form in $x_1,x_2$ gives the desired normal form.

A direct determinant computation shows that
\[
\det(Q_0+tQ_1) = -\frac14 t^2(1+at)(1+bt).
\]
For a general tangent pencil, the two residual roots are distinct and nonzero, so $a,b\neq 0$ and $a\neq b$; the proposition follows.
\end{proof}

\begin{proposition}\label{prop:unique-generic-boundary-orbit}
All pencils of the form $W_{a,b}$ with $a,b\in k^\times$ and $a\neq b$ lie in a single $\PGL(V)$-orbit.
\end{proposition}

\begin{proof}
The discriminant divisor of $W_{a,b}$ is
\[
2[0]+[-1/a]+[-1/b].
\]
The group $\PGL_2$ acts transitively on divisors of type $2+1+1$ on $\Pj^1$: such a divisor is determined by a labeled double point together with an unordered pair of distinct simple points, and $\PGL_2$ acts simply $3$-transitively on ordered triples of distinct points.

Now, let $W_{a,b}\in G$ be one of our pencils. Replacing the ordered basis $(Q_0,Q_1)$ of the same $2$-plane $W_{a,b}$ by another basis does not change the point of $G$; it merely reparameterizes the line $\ell_{W_{a,b}}\cong \Pj^1$. We may choose such a reparameterization so that the discriminant divisor becomes
\[
2[0]+[-1]+[-1/2].
\]
Next, we apply the coordinate changes from the proof of Lemma~\ref{lem:normal-form}. These transformations preserve the abstract discriminant divisor, because if $g\in \GL(V)$, then
\[
\det\bigl(s\,gQ_0g^\top+t\,gQ_1g^\top\bigr)=(\det g)^2\det(sQ_0+tQ_1).
\]
Therefore, the resulting normal form $W_{a',b'}$ still has discriminant divisor $2[0]+[-1]+[-1/2]$, so
\[
\{-1/a',-1/b'\}=\{-1,-1/2\}.
\]
Hence, $\{a',b'\}=\{1,2\}$. Exchanging $x_1$ and $x_2$ swaps $W_{1,2}$ and $W_{2,1}$, so all such pencils are $\PGL(V)$-equivalent to
\[
W_{\mathrm{node}}:=\big\langle x_0^2+x_1^2+x_2^2,\ x_0x_3+x_1^2+2x_2^2\big\rangle.
\]
This proves the claim.
\end{proof}

\begin{definition}
We denote the orbit of $W_{\mathrm{node}}$ by $\mathcal O_{\mathrm{node}}$.
\end{definition}

\begin{proposition}\label{prop:nodal-model}
The complete intersection $X_{W_{\mathrm{node}}}\subseteq \Pj^3$ has a unique singularity, and this singularity is an ordinary node. In particular, $X_{W_{\mathrm{node}}}$ is a quartic complete intersection of arithmetic genus $1$ and geometric genus $0$.
\end{proposition}

\begin{proof}
We work on the affine chart $x_3=1$ near the point $p=[0:0:0:1]$. There, the equations become
\[
 x_0^2+x_1^2+x_2^2=0,
 \qquad
 x_0+x_1^2+2x_2^2=0.
\]
Eliminating $x_0$ gives the local equation
\[
 x_1^2+x_2^2 + (\text{terms of order}\ge 4) =0.
\]
The quadratic term is nondegenerate, so the singularity at $p$ is an ordinary node. Next, to see that $p$ is the unique singular point, write
\[
F=x_0^2+x_1^2+x_2^2,
\qquad
G=x_0x_3+x_1^2+2x_2^2.
\]
A singular point of the complete intersection must satisfy $F=G=0$ and
\[
\operatorname{rank}
\begin{pmatrix}
2x_0 & 2x_1 & 2x_2 & 0\
x_3 & 2x_1 & 4x_2 & x_0
\end{pmatrix}<2.
\]
If the first row vanishes, then $x_0=x_1=x_2=0$, so the point is $p$. Otherwise, the two rows are proportional. The last column then makes $x_0=0$, and comparing the second and third columns shows that at most one of $x_1,x_2$ can be nonzero. The equations $F=G=0$ then force $x_1=x_2=0$, again giving $p$. Thus, $p$ is the unique singular point. Since a complete intersection of two quadrics in $\Pj^3$ has arithmetic genus $1$, and a single node lowers the geometric genus by one, the normalization has genus $0$.
\end{proof}

\subsection{Codimension estimates for the bad repeated-root loci}

We now prove that every repeated-root phenomenon other than generic simple tangency has codimension at least $2$ in $G$.

First, we establish the following straightforward lemma.

\begin{lemma}\label{lem:lines-meeting-singD}
The locus of lines $\ell\subseteq \Pj^9$ meeting $\Sing(D)$ has codimension at least $2$ in $G$.
\end{lemma}

\begin{proof}
The singular locus of $D$ is the locus of quadrics of rank at most $2$. For symmetric $4\times 4$ matrices, this locus has dimension $6$. Let us consider the incidence correspondence
\[
I_{\mathrm{sing}}:=\{([Q],\ell)\in \Sing(D)\times G:[Q]\in \ell\}.
\] 
For fixed $[Q]\in \Pj^9$, the family of lines through $[Q]$ is $\Pj^8$, so
\[
\dim I_{\mathrm{sing}} = 6+8 = 14.
\]
Its image in $G$ has dimension at most $14$, while $\dim G=16$. It follows that the locus of lines meeting $\Sing(D)$ has codimension at least $2$, as was sought.
\end{proof}

\begin{lemma}\label{lem:bad-tangent-directions}
The locus of lines tangent to $D$ at a smooth point but not of contact type $2+1+1$ has codimension at least $2$ in $G$.
\end{lemma}

\begin{proof}
Fix $[Q_0]\in D_{\mathrm{sm}}$, and consider tangent lines through $[Q_0]$. These are parametrized by $\Pj^7$. Along such a line, choose an affine parameter $t$ with $t=0$ corresponding to $[Q_0]$. Since the line is tangent to $D$, the restriction of the determinant quartic has the form
\[
f(t)= t^2(\alpha+\beta t+\gamma t^2),
\]
where $\alpha,\beta,\gamma$ depend as polynomials on the tangent direction.

The generic contact type is $2+1+1$, which occurs exactly when the residual quadratic
$\alpha+\beta t+\gamma t^2$
has two distinct roots, i.e. when
$\beta^2-4\alpha\gamma\neq 0.$
Thus, the non-generic tangent directions are contained in the hypersurface
$\{\beta^2-4\alpha\gamma=0\} \subseteq \Pj^7$, which is a proper containment because Lemma~\ref{lem:normal-form} exhibits explicit tangent directions of type $2+1+1$. Consequently, for fixed $[Q_0]$, the bad tangent directions have dimension at most $6$. Now, varying $[Q_0]$ over the $8$-dimensional variety $D_{\mathrm{sm}}$, we obtain a locus of dimension at most $8+6=14.$
This demonstrates that its image in $G$ has codimension at least $2$.
\end{proof}

\subsection{The structure of the special fiber}

\begin{proposition}\label{prop:T-irreducible}
The locus $T$ is an irreducible divisor in $G$.
\end{proposition}

\begin{proof}
Consider the incidence correspondence
\[
I:=\{([Q],\ell)\in D_{\mathrm{sm}}\times G : [Q]\in \ell\subseteq T_{[Q]}D\}.
\]
The smooth locus $D_{\mathrm{sm}}$ is irreducible of dimension $8$, since $D\subseteq \Pj^9$ is an irreducible quartic hypersurface. For fixed $[Q]\in D_{\mathrm{sm}}$, lines through $[Q]$ inside the tangent hyperplane $T_{[Q]}D\cong \Pj^8$ are parametrized by $\Pj^7$. Hence
\[
\dim I = 8+7=15.
\]
Since $I$ is a projective bundle over the irreducible variety $D_{\mathrm{sm}}$, it is itself irreducible.

By Lemmas~\ref{lem:lines-meeting-singD} and \ref{lem:bad-tangent-directions}, a general tangent line avoids $\Sing(D)$ and has contact type $2+1+1$ at a unique smooth point of $D$; in particular, the root type $2+2$, which would correspond to tangency at two smooth points, occurs only in codimension at least $2$. Therefore, the projection $I\to G$ is generically one-to-one onto its image. The image is exactly $T$, hence $T$ is irreducible of dimension $15$ (and so it is a divisor).
\end{proof}

All of this work amounts to the following result.

\begin{theorem}\label{thm:special-fiber-structure}
The repeated-root locus decomposes as
\[
F_\infty = T \cup Z,
\]
where $T$ is an irreducible divisor and $Z$ has codimension at least $2$ in $G$. In particular, $T$ is the unique divisorial component of $F_\infty$, and
\[
T = \overline{\mathcal O_{\mathrm{node}}}.
\]
\end{theorem}

\begin{proof}
Away from lines meeting $\Sing(D)$, a repeated root in the discriminant quartic is equivalent to non-transverse intersection of $\ell_W$ with $D$ at a smooth point, i.e. tangency. By Lemmas~\ref{lem:lines-meeting-singD} and \ref{lem:bad-tangent-directions}, all repeated-root phenomena other than generic simple tangency occur in codimension at least $2$. Therefore, the unique divisorial component of $F_\infty$ is the generic tangency locus $T$.

By Proposition~\ref{prop:unique-generic-boundary-orbit}, the general point of $T$ lies in the orbit $\mathcal O_{\mathrm{node}}$, and by Proposition~\ref{prop:T-irreducible} the divisor $T$ is irreducible. Therefore,
\[
T=\overline{\mathcal O_{\mathrm{node}}}.
\]
\end{proof}

\section{The class of the nodal divisor}

We now compute the class of the boundary divisor $T$.

\begin{theorem}\label{thm:nodal-class}
The nodal divisor has class
\[
[T]=12\sigma_1
\]
in $A^1(G)$.
\end{theorem}

\begin{proof}
Let us intersect $T$ with the same Schubert cycle $B\cong \Pj^1$ used in the proof of Theorem~\ref{thm:general-fiber-class}. Here, $B$ is the family of lines through a general point $q$ in a general plane $\Pi\subseteq \Pj^9$.

On this Schubert slice, the determinant hypersurface restricts to the smooth plane quartic
\[
C=\Pi\cap D.
\]
The intersection $B\cap T$ consists precisely of the lines through $q$ tangent to $C$. By Lemma~\ref{lem:12-tangents}, there are $12$ such lines, and for general $q$, all tangencies are simple; the requisite general-position statements may be justified, for instance, by Kleiman transversality \cite{Kleiman}. Equivalently, the projection $\pi_q\colon C\to \Pj^1$ has $12$ simple ramification points. At such a line, the discriminant of the corresponding binary quartic vanishes to first order, so the local equation of the tangency divisor cuts $B$ with multiplicity one. Thus, $B$ meets $T$ transversely at each of these $12$ points, and their intersection $B\cap T$ is reduced of degree $12$.

Write $[T]=m\sigma_1$. Since $\sigma_1\cdot \sigma_{8,7}=1$, we obtain
\[
m=[T]\cdot \sigma_{8,7}=12;
\]
therefore, $[T]=12\sigma_1$
\end{proof}

\section{Scheme-theoretic versus set-theoretic fibers}

Choudhary's codimension-one calculation in \cite{Choudhary} includes a separate discussion of the difference between scheme-theoretic fibers and reduced orbit closures. The same issue occurs here at only the CM values $j=0$ and $j=1728$. Away from those two points, the finite smooth fibers
are reduced; the special divisor $T$ is also generically reduced and appears with multiplicity one
in the fiber over $\infty$.

\begin{proposition}\label{prop:finite-stabilizer-smooth}
The stabilizer in $\PGL(V)$ of any smooth pencil is finite. In particular, every
$\PGL(V)$-orbit in $U_{\mathrm{sm}}$ has dimension $15$.
\end{proposition}

\begin{proof}
After Proposition~\ref{prop:diag}, a smooth pencil can be represented by
\[
W_\lambda=
\Big\langle
x_0^2+x_1^2+x_2^2+x_3^2,\
\lambda_0x_0^2+\lambda_1x_1^2+\lambda_2x_2^2+\lambda_3x_3^2
\Big\rangle
\]
with $\lambda_0,\lambda_1,\lambda_2,\lambda_3$ pairwise distinct.
Any element of $\PGL(V)$ stabilizing $W_\lambda$ preserves the finite set of singular
quadrics in the pencil, hence permutes the four distinguished rank-$3$ members corresponding
to the four roots of $\Delta_{W_\lambda}$. The kernel of this permutation action preserves
each of their vertices, so it preserves each coordinate line $\langle e_i\rangle$ and is
therefore represented by some diagonal matrix.

Now, for such a diagonal matrix to stabilize the pencil, there must exist $\alpha,\beta$
such that
\[
d_i^2=\alpha+\beta\lambda_i \qquad (i=0,1,2,3).
\]
Since the $\lambda_i$ are distinct, the four values $d_i^2$ are determined by the two
parameters $\alpha$ and $\beta$, and each $d_i$ is thereby determined up to sign. Thus, the diagonal kernel is finite. Since the image lies as a subgroup of the finite permutation group $S_4$ on four letters, the full stabilizer is finite.
\end{proof}

\begin{proposition}\label{prop:CM-ramification}
Let $\widetilde U_{\mathrm{sm}}\to U_{\mathrm{sm}}$ be the finite \'etale cover obtained by
ordering the four roots of the discriminant quartic. After using the $\PGL_2$-action on the
parameter line to send three ordered roots to $0,1,\infty$, the fourth root gives a coordinate
\[
\lambda \in \Pj^1\setminus\{0,1,\infty\},
\]
and the pullback of the map $p$ is the Legendre $j$-function
\[
j(\lambda)=2^8\frac{(1-\lambda+\lambda^2)^3}{\lambda^2(1-\lambda)^2}.
\]

The only finite branch values of this map are $0$ and $1728$. The ramification index is
$3$ over $0$ and $2$ over $1728$. Consequently, if
\[
O_a := \bigl(p^{-1}(a)\cap U_{\mathrm{sm}}\bigr)_{\mathrm{red}},
\]
then
\[
p^{-1}(a)\cap U_{\mathrm{sm}} = O_a
\qquad\text{for } a\notin\{0,1728,\infty\},
\]
while scheme-theoretically
\[
p^{-1}(1728)\cap U_{\mathrm{sm}} = 2\,O_{1728},
\qquad
p^{-1}(0)\cap U_{\mathrm{sm}} = 3\,O_0.
\]
Upon taking closures in $G$, one gets
\[
F_a=\overline{O_a}
\qquad\text{for } a\notin\{0,1728,\infty\},
\]
and
\[
F_{1728}=2\,\overline{O_{1728}},
\qquad
F_0=3\,\overline{O_0}
\]
as divisors on $G$.
\end{proposition}

\begin{proof}
Because the discriminant quartic is squarefree on $U_{\mathrm{sm}}$, the cover
$\widetilde U_{\mathrm{sm}}\to U_{\mathrm{sm}}$ obtained by ordering its four roots is finite
\'etale. The map $p$ depends only on the corresponding unordered divisor on $\Pj^1$,
so after passing to $\widetilde U_{\mathrm{sm}}$ and fixing three ordered roots at
$0,1,\infty$, the remaining transverse modulus is the Legendre parameter $\lambda$, and
$p$ is given by the displayed formula (see Section~2).

Since the cover is \'etale, the scheme-theoretic multiplicity of a fiber of $p$ may be checked
after pullback to $\widetilde U_{\mathrm{sm}}$, hence on the $\lambda$-line. Differentiating yields
\[
j'(\lambda)=
2^8\,
\frac{(\lambda-2)(\lambda+1)(2\lambda-1)(\lambda^2-\lambda+1)^2}
{\lambda^3(\lambda-1)^3}.
\]
It follows that the critical points in $\Pj^1\setminus\{0,1,\infty\}$ are
\[
\lambda=-1,\;\frac12,\;2,
\qquad\text{and}\qquad
\lambda^2-\lambda+1=0.
\]

At $\lambda=-1,\frac12,2$, we may check that $j(\lambda)=1728$, at which the zero $j'$ is simple. Thus, the ramification is simple there, so the ramification index over $1728$ is $2$.

At the two roots of $\lambda^2-\lambda+1=0$, the factor $1-\lambda+\lambda^2$ vanishes simply, while the denominator is nonzero, so $j(\lambda)$ vanishes to order $3$. Thus, the ramification index over $0$ is $3$.

There are no other finite critical values, so the finite smooth fibers are reduced away from
$0$ and $1728$, and have multiplicities $3$ and $2$ at those two values. Finally,
Theorem~\ref{thm:special-fiber-structure} shows that the only divisorial component of $G\setminus U_{\mathrm{sm}}$
lies over $\infty$, so taking closures in $G$ introduces no additional finite divisorial
components. This gives the asserted formulas for $F_a$.
\end{proof}

\begin{corollary}\label{cor:reduced-orbit-closure-classes}
The reduced codimension-one orbit closures in $G$ are
\[
\overline{O_a}\quad (a\notin\{0,1728,\infty\}),\qquad
\overline{O_{1728}},\qquad
\overline{O_0},\qquad
T.
\]
Their classes are
\[
[\overline{O_a}] = 12\sigma_1
\quad (a\notin\{0,1728,\infty\}),
\]
\[
[\overline{O_{1728}}] = 6\sigma_1,
\qquad
[\overline{O_0}] = 4\sigma_1,
\qquad
[T] = 12\sigma_1.
\]
That is, every divisorial fiber of $p$ has class $12\sigma_1$,
and the fibers over $1728$ and $0$ are nonreduced with multiplicities $2$ and $3$.
\end{corollary}

\begin{proof}
For every finite smooth value $a$, Theorem~\ref{thm:general-fiber-class} gives
\[
[F_a]=12\sigma_1.
\]
For $a\notin\{0,1728,\infty\}$, Proposition~\ref{prop:CM-ramification} gives
$F_a=\overline{O_a}$, so $[\overline{O_a}]=12\sigma_1$.

For the CM values, Proposition~\ref{prop:CM-ramification} gives
\[
F_{1728}=2\,\overline{O_{1728}},
\qquad
F_0=3\,\overline{O_0},
\]
hence
\[
[\overline{O_{1728}}]=\frac{1}{2}[F_{1728}]=6\sigma_1,
\qquad
[\overline{O_0}]=\frac{1}{3}[F_0]=4\sigma_1.
\]

Finally, Theorem~\ref{thm:nodal-class} gives $[T]=12\sigma_1$, and the computation in
the proof of Theorem~\ref{thm:nodal-class} already showed that $T$ appears with multiplicity
one in the special fiber.
\end{proof}

The nodal divisor appears with multiplicity one in $F_\infty$. To see this, we may resolve the rational map $p\colon G\dashrightarrow \Pj^1$ to a morphism
$\widetilde p\colon \widetilde G\to \Pj^1$.
All fibers of $\widetilde p$ are linearly equivalent on $\widetilde G$, and the exceptional locus of the birational map $\widetilde G\to G$ lies over the indeterminacy locus of $p$, which is contained in $G\setminus U_{\mathrm{sm}}$ and contributes only codimension at least $2$ upon pushforward. Therefore, the divisor classes of the pushed-forward fibers agree:
$[F_\infty]=[F_a]=12\sigma_1$.
Writing the special fiber as
$[F_\infty]=m[T]+(\text{higher-codimension terms})$,
Theorem~\ref{thm:nodal-class} then hands us
$12\sigma_1=[F_\infty]=m[T]=12m\sigma_1$,
so $m=1$.

\section{Proof of the main theorem}

We now assemble all the pieces of the codimension-one puzzle.

\begin{proof}[Proof of Theorem~\ref{thm:main}]
The existence of the rational $j$-map $p$ and the description of its smooth fibers follow from Proposition~\ref{prop:orbit-classification} and Corollary~\ref{cor:fibers-are-orbits}. The formula for the Chow class
\[
[F_a]=12\sigma_1
\]
for scheme-theoretic fibers is Theorem~\ref{thm:general-fiber-class}. The structure of the special fiber is given by Theorem~\ref{thm:special-fiber-structure}, and the explicit nodal model is Proposition~\ref{prop:nodal-model}. Theorem~\ref{thm:nodal-class} shows that the nodal divisor has class $12\sigma_1$. The distinction between scheme-theoretic fibers and reduced orbit closures at the CM values, and the resulting classes of the reduced orbit closures, are given by Proposition~\ref{prop:CM-ramification} and Corollary~\ref{cor:reduced-orbit-closure-classes} respectively. 

Thus, the codimension-one reduced orbit closures are the closures $\overline{O_a}$ of the smooth
fibers, together with the nodal boundary divisor $T$. For $a\notin\{0,1728,\infty\}$ we obtain
$[\overline{O_a}]=12\sigma_1$, while
\[
[\overline{O_{1728}}]=6\sigma_1,\qquad
[\overline{O_0}]=4\sigma_1,\qquad
[T]=12\sigma_1.
\]
On the level of scheme-theoretic fibers, every divisorial fiber of $p$ has class $12\sigma_1$;
the only nonreduced finite fibers are
\[
F_{1728}=2\,\overline{O_{1728}},\qquad F_0=3\,\overline{O_0}.
\]
\end{proof}

\section{Further directions}

Several natural problems and generalizations emerge from this calculation.
\begin{enumerate}[label=\textup{(\roman*)}]
    \item What we have described is the first clear contour of this case; beyond it, the repeated-root locus breaks into a more intricate archipelago. Determine the higher-codimension orbit closures inside the repeated-root locus, including the root types $2+2$, $3+1$, and $4$, as well as the loci where the tangent point lies in $\Sing(D)$.
    \item Interpret the nodal divisor and its degenerations directly in terms of the projective geometry of complete intersections of two quadrics in $\Pj^3$, and compare this with Choudhary's two-component special fiber for nets of conics. In particular, it would be interesting to isolate precisely why tangency along $\Sing(D)$ is codimension at least $2$ here, whereas the analogous phenomenon in Choudhary's setting remains divisorial.
    \item Describe the stabilizers along the smooth locus, especially at the CM values $j=0$ and $j=1728$, and relate them to the orbifold structure of the coarse moduli line.
\end{enumerate}

\end{document}